\documentclass{amsart}
\usepackage[utf8]{inputenc}
\usepackage{latexsym,amssymb,amsmath,amsthm,amscd,graphicx}
\usepackage{amsfonts}
\usepackage{xcolor}
\usepackage[english]{babel}
\usepackage[a4paper,bindingoffset=0.2in,%
            left=1in,right=1in,top=1in,bottom=1in,%
            footskip=.5in]{geometry}
\newtheorem{theorem}{Theorem}[section]

\begin{document}

\title[Generalized Von Neumann-Jordan constant for Morrey spaces]{Generalized Von Neumann-Jordan
Constant for Morrey Spaces and Small Morrey Spaces}

\author[H. Rahman]{Hairur Rahman}
\address{Departement of Mathematics, Islamic State University Maulana Malik Ibrahim Malang,
Jalan Gajayana No.~50, Indonesia}
\email{hairur@mat.uin-malang.ac.id}


\author[H. Gunawan]{Hendra Gunawan}
\address{Analysis and Geometry Group, Faculty of Mathematics and Natural Sciences, Bandung
Institute of Technology, Bandung 40132, Indonesia}
\email{hgunawan@math.itb.ac.id}



\begin{abstract}
In this paper we calculate some geometric constants for Morrey spaces and small Morrey spaces,
namely generalized Von Neumann-Jordan constant, modified Von Neumann-Jordan constants,
and Zb\'{a}ganu constant. All these constants measure the uniformly nonsquareness of the spaces.
We obtain that their values are the same as the value of Von Neumann-Jordan constant for Morrey
spaces and small Morrey spaces.
\end{abstract}

\maketitle

\lineskip=12pt

\section{Introduction}

We shall discuss some geometric constants for Banach spaces. Three geometric constants have been
studied by Gunawan {\it et al.} \cite{Gun}, but there are other geometric constants which were
introduced by other authors. Inspired by Clarkson \cite{Clar}, who introduced the Von Neumann-Jordan
constant for a Banach space $(X,\|\cdot\|_X)$, Cui {\it et al.}~\cite{Cui} defined the generalized
Von Neumann-Jordan constant $C_{NJ}^{(s)}(X)$ by
$$
C_{NJ}^{(s)}(X) := \sup \left\lbrace \frac{\Vert x+y\Vert_X^s + \Vert x-y\Vert_X^s}{2^{s-1}
(\Vert x\Vert_X^s +\Vert y\Vert_X^s)} : x,y\in X\backslash \lbrace 0\rbrace\right\rbrace ,
$$
for $1\le s<\infty$. Observe that $1\leq C_{NJ}^{(s)}(X) \leq 2$ for any Banach space $X$.

Several years earlier, Alonso {\it et al.}~\cite{Alon} and Gao \cite{Gao} studied the modified
Von Neumann-Jordan constant, which is defined by
$$
C'_{NJ}(X) := \sup \left\lbrace \frac{\Vert x+y\Vert_X^2 + \Vert x-y\Vert_X^2}{4} : x,y\in X,
\Vert x\Vert_X = \Vert y\Vert_X = 1 \right\rbrace.
$$
Note that $1\leq C'_{NJ}(X) \leq C_{NJ}(X)\leq 2$ is for any Banach space $X$.
This constant was generalized by Yang {\it et al.}~\cite{Yang} to the following constant
$$
\bar{C}_{NJ}^{(s)}(X) := \sup \left\lbrace \frac{\Vert x+y\Vert_X^s + \Vert x-y\Vert_X^s}{2^s} :
x,y\in X, \Vert x\Vert_X = \Vert y\Vert_X = 1 \right\rbrace.
$$
It is proved in \cite{Yang} that $\bar{C}_{NJ}^{(s)}(X)\leq C_{NJ}^{(s)}(X) \leq 2^{s-1}\left[1+\left(
2^{\frac{1}{s}}\left(C'_{NJ}(X)\right)^{\frac{1}{s}} -1\right)\right]^{s-1}$ for any Banach space $X$.

Beside the above constants, we also know a constant called Zb\'{a}ganu constant, which is defined by
$$
C_Z(X) := \sup \left\lbrace \frac{\Vert x+y\Vert_X \Vert x-y\Vert_X}{\Vert x\Vert_X^2 +\Vert y\Vert_X^2} :
x,y\in X\backslash \lbrace 0\rbrace \right\rbrace.
$$
This constant was introduced by Zb\'{a}ganu \cite{Ganu} and studied further in \cite{Lor,Lor1}. It is
easy to see that $1\leq C_Z(X) \leq C_{NJ}(X)\leq 2$ for any Banach space X.

In this paper, we want to calculate the values of those constants for Morrey spaces as well as for
small Morrey spaces.

Let $1\leq p\leq q<\infty$. The Morrey spaces $\mathcal{M}^p_q = \mathcal{M}^p_q(\mathbb{R}^n)$ is
the set of all measurable function $f$ such that
$$
\Vert f\Vert_{\mathcal{M}^p_q} :=\sup_{a\in \mathbb{R}^n , r>0} \vert B(a,r)\vert^{\frac{1}{q} -\frac{1}{p}}
\left(\int_B \vert f(x)\vert^p dx\right)^{\frac{1}{p}}
$$
is finite. Here $B(a,r) = \lbrace x\in \mathbb{R}^n :\vert x-a\vert <r \rbrace$ dan $|B(a,r)|$ denotes its
Lebesgue measure. Since $\mathcal{M}^p_q$ is a Banach space, it follows from \cite{Gao1} that
$$
C_{NJ}^{(s)}(\mathcal{M}^p_q),\ C'_{NJ}(\mathcal{M}^p_q),\ \bar{C}_{NJ}^{(s)}(\mathcal{M}^p_q),\ C_Z(\mathcal{M}^p_q)\leq 2.
$$
Meanwhile, the small Morrey $m^p_q=m^p_q(\mathbb{R}^n)$ is the set of all measurable function $f$ such that
$$
\Vert f\Vert_{m^p_q} := \sup_{a\in \mathbb{R}^n , r\in (0,1)} \vert B(a,r)\vert^{\frac{1}{q} -\frac{1}{p}}
\left(\int_B \vert f(x)\vert^p dx\right)^{\frac{1}{p}}
$$
is finite. Small Morrey spaces are also Banach spaces \cite{Saw} and, for all $p$ and $q$, the small Morrey spaces
properly contain the Morrey spaces.

The values of Von Neumann-Jordan constants for Morrey spaces and small Morrey spaces have been
computed by Gunawan {\it et al.} \cite{Gun} and Mu'tazili and Gunawan in \cite{Mut}, respectively.
In this paper, we want to calculate other geometric constants defined above.
Our results are presented in the following sections.

\section{Main Results}

The values of the geometric constants defined in the previous section for Morrey spaces and small
Morrey spaces are presented in the following theorems.

\begin{theorem}
\label{Teo1}
For $1\leq p < q < \infty$ and $1\le s<\infty$, we have
$$
C_{NJ}^{(s)}(\mathcal{M}^p_q) = C'_{NJ}(\mathcal{M}^p_q) = \bar{C}_{NJ}^{(s)}(\mathcal{M}^p_q)=
C_Z(\mathcal{M}^p_q) = 2.
$$
\end{theorem}

\begin{proof}
To prove the theorem, we consider the following functions: $f(x)=\vert x\vert^{-\frac{n}{q}},\
g(x)=\chi_{(0,1)}(\vert x\vert) f(x)$, $h(x)=f(x)-g(x),$ and $k(x)=g(x)-h(x)$, where $x\in \mathbb{R}^n$.
All functions are radial functions. Note that $f \in \mathcal{M}^p_q$, with
\[
\Vert f\Vert_{\mathcal{M}^p_q} = c_n\left(1-\frac{p}{q}\right)^{-\frac{1}{p}}.
\]
One may also observe that
$$
\Vert f\Vert_{\mathcal{M}^p_q} = \Vert g\Vert_{\mathcal{M}^p_q} = \Vert h\Vert_{\mathcal{M}^p_q} =
\Vert k\Vert_{\mathcal{M}^p_q}.
$$
It thus follows that
\begin{align*}
C_{NJ}^{(s)}(\mathcal{M}^p_q) &\geq \frac{\Vert f+k\Vert^s_{\mathcal{M}^p_q} +
\Vert f-k\Vert^s_{\mathcal{M}^p_q}}{2^{s-1} (\Vert f\Vert^s_{\mathcal{M}^p_q} +\Vert k\Vert^s_{\mathcal{M}^p_q})}\\
&=\frac{\Vert 2g\Vert^s_{\mathcal{M}^p_q} + \Vert 2h\Vert^s_{\mathcal{M}^p_q}}{2^{s-1}\left(\Vert f\Vert^s_{\mathcal{M}^p_q}
+\Vert f\Vert^s_{\mathcal{M}^p_q}\right)}\\
&=\frac{2^s\Vert g\Vert^s_{\mathcal{M}^p_q} + 2^s\Vert h\Vert^s_{\mathcal{M}^p_q}}{2^{s-1}\left(2\Vert f\Vert^s_{\mathcal{M}^p_q}\right)}\\
&= 2.
\end{align*}
Since $C_{NJ}^{(s)}(\mathcal{M}^p_q) \leq 2$, we can conclude that $C_{NJ}^{(s)}(\mathcal{M}^p_q) = 2$.

Now, for the modified Von Neumann-Jordan constant, we consider $\frac{f}{\Vert f\Vert_{\mathcal{M}^p_q}}$ and
$\frac{k}{\Vert k\Vert_{\mathcal{M}^p_q}}=\frac{k}{\Vert f\Vert_{\mathcal{M}^p_q}}$. We have
\begin{align*}
C'_{NJ}(\mathcal{M}^p_q) &\geq \frac{\Vert f+k\Vert_{\mathcal{M}^p_q}^2 + \Vert f-k\Vert_{\mathcal{M}^p_q}^2}{4\Vert f\Vert^2_{\mathcal{M}^p_q}}\\
&=\frac{\Vert 2g\Vert^2_{\mathcal{M}^p_q} + \Vert 2h\Vert^2_{\mathcal{M}^p_q}}{4\Vert f\Vert^2_{\mathcal{M}^p_q}}\\
&=\frac{4\left(\Vert g\Vert^2_{\mathcal{M}^p_q} + \Vert h\Vert^2_{\mathcal{M}^p_q}\right)}{4\Vert f\Vert^2_{\mathcal{M}^p_q}}\\
&= 2.
\end{align*}
This shows that $C'_{NJ}(\mathcal{M}^p_q) =2$.

Similarly, for generalized modified Von Neumann-Jordan constant, we have
\begin{align*}
\bar{C}_{NJ}^{(s)}(\mathcal{M}^p_q) &\geq \frac{\Vert f+k\Vert^s_{\mathcal{M}^p_q} +
\Vert f-k\Vert^s_{\mathcal{M}^p_q}}{2^s\Vert f\Vert^s_{\mathcal{M}^p_q}}\\
&=\frac{\Vert 2g\Vert^s_{\mathcal{M}^p_q} + \Vert 2h\Vert^s_{\mathcal{M}^p_q}}{2^s\Vert f\Vert^s_{\mathcal{M}^p_q}}\\
&=\frac{2^s\Vert g\Vert^s_{\mathcal{M}^p_q} + 2^s\Vert h\Vert^s_{\mathcal{M}^p_q}}{2^s\Vert f\Vert^s_{\mathcal{M}^p_q}}\\
&= 2.
\end{align*}
This also leads us to the conclusion that $\bar{C}_{NJ}^{(s)}(\mathcal{M}^p_q) = 2$.

Last, for the Zb\'{a}ganu constant, we have
\begin{align*}
C_Z(\mathcal{M}^p_q) &\geq \frac{\Vert f+k\Vert_{\mathcal{M}^p_q} \Vert f-k\Vert_{\mathcal{M}^p_q}}{\Vert
f\Vert_{\mathcal{M}^p_q}^2 +\Vert k\Vert_{\mathcal{M}^p_q}^2}\\
&=\frac{\Vert 2g\Vert_{\mathcal{M}^p_q} \Vert 2h\Vert_{\mathcal{M}^p_q}}{2\Vert f\Vert_{\mathcal{M}^p_q}^2}\\
&= 2,
\end{align*}
which implies that $C_Z(\mathcal{M}^p_q) = 2$, as desired.
\end{proof}

For small Morrey spaces, we have the following theorem.

\begin{theorem}
Let $1\leq p<q<\infty$ and $1\le s<\infty$. Then
$$
C_{NJ}^{(s)}(m^p_q) = C'_{NJ}(m^p_q) = \bar{C}_{NJ}^{(s)}(m^p_q)= C_Z(m^p_q) = 2.
$$
\end{theorem}

\begin{proof}
The idea of the proof is similar to that of Theorem \ref{Teo1}, but we use different functions. For $\varepsilon \in(0,1)$,
we consider $f(x)=\chi_{(0,1)}(|x|)\vert x\vert^{-\frac{n}{q}},\ g(x)=\chi_{(0,\varepsilon)}(\vert x\vert) f(x),\
h(x)=f(x)-g(x),$ and $k(x)=g(x)-h(x)$, where $x\in \mathbb{R}^n$. Here $g$ depends on $\varepsilon$, so that
$h$ and $k$ also depend on $\varepsilon$. Since all functions are radial functions, it is not hard to compute their
norms. First, we obtain that
\[
\Vert f\Vert_{m^p_q} = \Vert g\Vert_{m^p_q}= \Vert k\Vert_{m^p_q} = c_n \left(1-\frac{p}{q}\right)^{-\frac{1}{p}}.
\]
Next, we observe that
\begin{align*}
\Vert h\Vert_{m^p_q}&= \sup_{a\in \mathbb{R}^n, r\in (0,1)} \vert B(a,r)\vert^{\frac{1}{q} -\frac{1}{p}}
\left(\int_{B(a,r)} \vert f(x)\left(1-\chi_{(0,1)} (\vert x\vert)\right)\vert^p dx\right)^{\frac{1}{p}}\\
&\geq \sup_{r\in (0,1)} Cr^{n\left(\frac{1}{q} -\frac{1}{p}\right)} \left(\int_{B(0,r)} \vert x\vert^p
\left(1-\chi_{(0,1)} (\vert x\vert)\right) dx\right)^{\frac{1}{p}}\\
&= \sup_{r\in (\varepsilon,1)} Cr^{n\left(\frac{1}{q} -\frac{1}{p}\right)} \left(\int_{B(\varepsilon,r)}
\vert x\vert^p dx\right)^{\frac{1}{p}}\\
&= \sup_{r\in (\varepsilon,1)} Cr^{n\left(\frac{1}{q} -\frac{1}{p}\right)} \left(\int_{\varepsilon}^r
r^{\frac{-np}{q}} r^{n-1} dr\right)^{\frac{1}{p}}\\
&= \sup_{r\in (\varepsilon,1)} C\left[n\left(1-\frac{p}{q}\right)\right]^{-\frac{1}{p}} \left(1-r^{\frac{np}{q}-n}
\varepsilon^{n-\frac{np}{q}}\right)^{\frac{1}{p}}\\
&= \Vert f\Vert_{m^p_q} \left(1-\varepsilon^{n-\frac{np}{q}}\right)^{\frac{1}{p}}.
\end{align*}

We can now calculate the constants. First, let us observe the generalized Von Neumann-Jordan constant:
\begin{align*}
C_{NJ}^{(s)}(m^p_q) &\geq \frac{\Vert f+k\Vert^s_{m^p_q} + \Vert f-k\Vert^s_{m^p_q}}{2^{s-1}
\left(\Vert f\Vert^s_{m^p_q} +\Vert k\Vert^s_{m^p_q}\right)}\\
&= \frac{\Vert 2g\Vert^s_{m^p_q} + \Vert 2h\Vert^s_{m^p_q}}{2^{s-1} \left(2\Vert f\Vert^s_{m^p_q}\right)}\\
&\ge \frac{2^s\Vert f\Vert^s_{m^p_q} + 2^s\Vert f\Vert^s_{m^p_q}\left(1-\varepsilon^{n-\frac{np}{q}}\right)^{\frac{s}{p}}}{2^{s}
\Vert f\Vert^s_{m^p_q}}\\
&= 1+(1-\varepsilon^{n-\frac{np}{q}})^{\frac{s}{p}}.
\end{align*}
Since we may choose $\varepsilon$ to be arbitrary small, we conclude that $C_{NJ}^{(s)}(m^p_q) = 2$ (for we know that
the constant cannot be larger than 2).

We now move to the modified Von Neumann-Jordan constant. Noting that $\|f\|_{m^p_q}=\|k\|_{m^p_q}$,
we consider $\frac{f}{\Vert f\Vert_{m^p_q}}$ and $\frac{k}{\Vert f\Vert_{m^p_q}}$. We have
\begin{align*}
C'_{NJ}(m^p_q)&\ge\frac{\Vert f+k\Vert_{m^p_q}^2 + \Vert f-k\Vert_{m^p_q}^2}{4\Vert f\Vert^2_{m^p_q}}\\
&= \frac{\Vert 2g\Vert^2_{m^p_q} + \Vert 2h\Vert^2_{m^p_q}}{4\Vert f\Vert^2_{m^p_q}}\\
&\ge \frac{4\Vert f\Vert^2_{m^p_q} + 4\Vert f\Vert^2_{m^p_q}\left(1-\varepsilon^{n-\frac{np}{q}}\right)^{\frac{2}{p}}}{4\Vert f\Vert^2_{m^p_q}}\\
&= 1+(1-\varepsilon^{n-\frac{np}{q}})^{\frac{2}{p}}.
\end{align*}
By using similiar arguments as above, we conclude that $C'_{NJ}(m^p_q) = 2$.

For the generalization of the modified Von Neumann-Jordan constant, we observe that
\begin{align*}
\bar{C}_{NJ}^{(s)}(m^p_q)
&\ge \frac{\Vert f+k\Vert^s_{m^p_q} + \Vert f-k\Vert^s_{m^p_q}}{2^s\Vert f\Vert^s_{m^p_q}}\\
&= \frac{\Vert 2g\Vert^s_{m^p_q} + \Vert 2h\Vert^s_{m^p_q}}{2^s\Vert f\Vert^s_{m^p_q}}\\
&\ge \frac{2^s\Vert f\Vert^s_{m^p_q} + 2^s\Vert f\Vert^s_{m^p_q}\left(1-\varepsilon^{n-\frac{np}{q}}\right)^{\frac{s}{p}}}{2^s\Vert f\Vert^s_{m^p_q}}\\
&= 1+(1-\varepsilon^{n-\frac{np}{q}})^{\frac{s}{p}}.
\end{align*}
With the same arguments as above, we conclude that $\bar{C}_{NJ}^{(s)}(m^p_q) = 2$.

Last, for the Zb\'{a}ganu constant, we have
\begin{align*}
C_Z(m^p_q) &\geq \frac{\Vert f+k\Vert_{m^p_q} \Vert f-k\Vert_{m^p_q}}{\Vert f\Vert_{m^p_q}^2 +\Vert k\Vert_{m^p_q}^2}\\
&= \frac{4\Vert g\Vert_{m^p_q} \Vert h\Vert_{m^p_q}}{2\Vert f\Vert_{m^p_q}^2}\\
&\ge \frac{2\Vert f\Vert_{m^p_q}^2 \left(1-\varepsilon^{n-\frac{np}{q}}\right)^{\frac{1}{p}}}{\Vert f\Vert_{m^p_q}^2}\\
&= 2\left(1-\varepsilon^{n-\frac{np}{q}}\right)^{\frac{1}{p}},
\end{align*}
By the same arguments, we conclude that $C_Z(m^p_q) = 2$.
\end{proof}

{\parindent=0cm
\textbf{Acknowledgement.} This first author is supported by UIN Maliki Research and Innovation
Program 2019. The second author is supported by ITB Research and Innovation Program 2020. 
}

\end{document}